\documentclass[11pt,letterpaper]{article}
\usepackage{amsfonts, amsmath, amssymb, amscd, amsthm, graphicx, mathrsfs, wasysym, setspace, mdwlist, color, transparent, import}
\usepackage{hyperref}
\usepackage{url}
\makeatletter
\def\blfootnote{\xdef\@thefnmark{}\@footnotetext}
\makeatother

\newtheorem{thm}{Theorem}[section]

\newtheorem{lem}[thm]{Lemma}
\newtheorem{prop}[thm]{Proposition}
\theoremstyle{definition}
\newtheorem{defn}[thm]{Definition}

\theoremstyle{remark}
\newtheorem{rem}[thm]{Remark}
\newtheorem{ex}[thm]{Example}
\newcommand{\wreath}[2]{#1 \, {\rm wr} \, #2}

\begin{document}

\title{A mixed identity-free elementary amenable group}
\author{B. Jacobson}
\date{}

\maketitle

\begin{abstract}
A group $G$ is called \textit{mixed identity-free} if for every $n \in \mathbb{N}$ and every $w \in G \ast F_n$ there exists a homomorphism $\varphi: G \ast F_n \rightarrow G$ such that $\varphi$ is the identity on $G$ and $\varphi(w)$ is nontrivial. In this paper, we make a modification to the construction of elementary amenable lacunary hyperbolic groups provided by Ol'shanskii, Osin, and Sapir in \cite{O-O-S} to produce finitely generated elementary amenable groups which are mixed identity-free. As a byproduct of this construction, we also obtain locally finite $p$-groups which are mixed identity-free.
\end{abstract}

\vspace{-2mm} \hspace{3mm} \textbf{MSC Subject Classification:} 20F65, 20F67.



\section{Introduction}


Let $F=F(x,y,\ldots)$ denote the free group with basis $\{x,y,\ldots\}$. A group $G$ is said to satisfy the \textit{mixed identity} $w=1$ for some $w \in  G \ast F$ if for every homomorphism $\varepsilon : G \ast F \rightarrow G$ which is identical on $G$, the image of $w$ under $\varepsilon$ is trivial. The mixed identity $w=1$ is called \textit{nontrivial} if $w$ is a nontrivial element of $G \ast F$. Mixed identities may be viewed as generalizations of the usual group identites. For example, while any abelian group satisfies the nontrivial identity $[x,y]=1$, any group $G$ with nontrivial center $Z(G)$ satisfies the nontrivial mixed identity $[x,g]=1$ for any element $g \in Z(G) \backslash \{1\}$. Importantly, a group can satisfy a nontrivial mixed identity while failing to satisfy any nontrivial identity, as shown below.

\begin{ex}\label{direct product ex}
Let $A$ and $B$ be nontrivial groups. Then direct product $G=A \times B$ satisfies the mixed identity $[[x,a],b]=1$ which is nontrivial for any choice of $a \in A \backslash \{1\}$ and $b \in B \backslash \{1\}$. If, furthermore, $A$ does not satisfy any nontrivial identity, then $G$ also fails to satisfy any nontrivial identity, since if $w=1$ holds in $G$ for some $w \in F$, it must also hold in any subgroup of $G$.
\end{ex}

$G$ is called \textit{mixed identity-free} (hereafter abbreviated \textit{MIF}) if it does not satisfy any nontrivial mixed identity. Groups that are MIF are subject to strict structural restrictions. For example, MIF groups do not decompose as nontrivial direct products (as demonstrated above) and have infinite conjugacy class property. For finitely generated groups, the property of being MIF implies infinite girth. MIF groups also resemble free products from the model theoretic point of view, i.e. a countable group $G$ is MIF if and only if $G$ and $G \ast F_n$ are universally equivalent as $G$-groups for all $n \in \mathbb{N}$. (For proofs, see \cite[Prop. 5.3, 5.4]{H-O}.)

In \cite{H-O}, Hull and Osin call for an example of a finitely generated MIF amenable group and suggest that the elementary amenable lacunary hyprbolic groups constructed in \cite{O-O-S} are reasonable candidates. In this paper, we show that a modification to this construction yields the following.

\begin{thm}\label{main result}
There exists a $2$-generated elementary amenable group which is MIF.
\end{thm}

\noindent Moreover, as a byproduct of the construction, we obtain the following:

\begin{thm}\label{loc nilp}
For each prime $p$, there exists a locally finite $p$-group which is MIF.
\end{thm}

\noindent We end the paper by examining two other reasonable candidates for examples of finitely generated MIF amenable groups, the Grigorchuk group and the identity-free amenable groups of infinite girth constructed by Akhmedov in \cite{A}, and show that in each case, the group in question satisfies a nontrivial mixed identity.


\section{Preliminaries}


\subsection{Mixed identities}

Given a group $G$, the following remark allows us to simplify the problem of showing that $G$ is MIF by considering only mixed identities arising from elements of $G \ast \langle x \rangle$.

\begin{rem}
Observe that for any nontrivial element $g \in G$, the extension of the identity map on $G$ to a map $$\iota: G \ast F( x_1, x_2, \ldots ) \rightarrow G \ast \langle x \rangle$$ given by sending $x_i \mapsto x^i g x^i$ is an embedding. As a result, $G$ satisfies a nontrivial mixed identity $u=1$ for some $u \in G \ast F ( x_1, \ldots, x_n )$ if and only if $G$ satisfies a nontrivial (single-variable) mixed identity $v=1$ for some $v \in G \ast \langle x \rangle$.
\end{rem}

It will be helpful when considering single-variable mixed identities to have the following notation. Given an element $w(x) \in G \ast \langle x \rangle$ and an element $g \in G$, let $w(g)$ denote the image of $w(x)$ in $G$ under the homomorphism $G \ast \langle x \rangle \rightarrow G$ given by taking the identity map on $G$ and sending $x \mapsto g$.

\subsection{An outline of the Ol'shanskii-Osin-Sapir construction}\label{OOS construction}

\begin{defn}
Let $G$ and $H$ be groups with a homomorphism $\varepsilon : G \rightarrow H$, and let $S$ a generating set for $G$. Provided $\varepsilon$ is not injective, the \textit{injectivity radius} of $\varepsilon$ with respect to $S$ is the maximal radius of a ball in the Cayley graph $\Gamma (G,S)$ about $1_G$ on which $\varepsilon$ is injective. If $\varepsilon$ is injective, then the injectivity radius of $\varepsilon$ is infinity.
\end{defn}

\noindent A group $G$ is called \textit{lacunary hyperbolic} if at least one of its asymptotic cones is an $\mathbb{R}$-tree. Equivalently (via \cite[Thm. 3.3]{O-O-S}),

\begin{defn}\label{lac hyp defn} A group $G$ is called \textit{lacunary hyperbolic} if there exists of groups $G_i$ with corresponding finite generating sets $S_i$ and epimorphisms $\varepsilon_i : G_i \twoheadrightarrow G_{i+1}$ such that the following conditions hold.
\begin{itemize}
\item[(i)] $G$ is the direct limit of the sequence
\begin{equation*}
G_0 \overset{\varepsilon_0} \twoheadrightarrow G_1 \overset{\varepsilon_1} \twoheadrightarrow G_2
\overset{\varepsilon_2} \twoheadrightarrow \ldots
\end{equation*}
\item[(ii)] $\varepsilon(S_i)=S_{i+1}$.
\item[(iii)] Each $G_i$ is $\delta_i$-hyperbolic with respect to the generating set $S_i$.
\item[(iv)] Let $r(\varepsilon_i)$ denote the injectivity radius of $\varepsilon_i$ with respect to the generating set $S_i$. Then $\delta_i=o(r(\varepsilon_i))$.
\end{itemize}
\end{defn}

\noindent In \cite{O-O-S}, Ol'shanskii, Osin, and Sapir provide a construction of elementary amenable groups satisfying the direct limit characterization of lacunary hyperbolicity above. As the main result of this paper arises as a modification of this construction, we present a brief overview the construction below, highlighting the key ingredients necessary to make this modification.

\begin{rem}
It should be noted that while the goal of the construction in \cite{O-O-S} is a direct limit satisfying the conditions of Definition \ref{lac hyp defn}, the guarantee of these conditions is an optional element of the direct limit construction in this paper. (See Remark \ref{lac hyp rem} for a modification of the construction in this paper which ensures that these conditions are met so that the resulting group is lacunary hyperbolic.)
\end{rem}

Given a prime number $p$ and a non-decreasing sequence $\textbf{c}$ of natural numbers $c_1 \leq c_2 \leq \ldots$, define the following. Let $A_0 = A_0 (p)$ be the group $$A_0 = \langle a_i, \, i \in \mathbb{Z} \, | \, a_i ^p =1, \, i \in \mathbb{Z}\rangle$$ and let $R_n$ denote the collection of relations of the form $$[\ldots[a_{i_0},a_{i_1}],\ldots,a_{i_{c_n}}]=1$$ for all commutators with $\max\limits_{j,k} |i_j - i_k| \leq n$. For $n \in \mathbb{N}$, define $A_n=A_n(p,c_1, \ldots,c_n)$ by $$A_n = \langle A_0 \, | \, \cup_{i=1} ^n R_i\rangle$$ or equivalently, $$A_n = \langle A_{n-1} \, | \, R_n \rangle.$$ Let $A=A(p,\textbf{c})$ be the group $$A = \langle A_0 \, | \, \cup_{i=1} ^\infty R_i\rangle.$$ Notice firstly that the group $A$ is a locally nilpotent group generated by elements of order $p$ and is thus a locally finite $p$-group. Indeed, every finitely generated subgroup of $A$ is contained in a subgroup $B= \langle a_{-N}, \ldots, a_{N} \rangle$ for some $N \in \mathbb{N}$. Since $B$ is a nilpotent group generated by finitely many elements of finite order, it is a finite group, and moreover, since all of its (nontrivial) generators are of order $p$, $B$ is its own unique Sylow $p$-subgroup and is thus a $p$ group.

Secondly, observe that $A_n$ and $A$ admit automorphisms $\varphi_n$ and $\varphi$ respectively, both given by extending the map on generators $a_i \rightarrow a_{i+1}, \, i \in \mathbb{Z}$. Define the group $G=G(p,\textbf{c})$ to be the extension of $A$ by $\varphi$, i.e. $$G= \langle A, t \, | \, t a_i t^{-1} = \varphi(a_i), i \in \mathbb{Z} \rangle.$$ Then $G$ is clearly $2$-generated ($G=\langle a_0, t \rangle$), and since $G$ is (locally finite $p$-group)-by-(infinite cyclic), it is elementary amenable.

Similarly, for $n \in \mathbb{N} \cup \{0\}$, define the group $G_n=G_n(p,c_1,\ldots,c_n$) to be
\begin{equation}\label{Gn}
G_n= \langle A_n, t \, | \, t a_i t^{-1} = \varphi_n(a_i), i \in \mathbb{Z} \rangle
\end{equation}
and observe that for each $n$, the natural quotient map $A_n \twoheadrightarrow A_{n+1}$ extends via the identity on $t$ to a map $\varepsilon_n : G_n \twoheadrightarrow G_{n+1}$ so that the group $G$ is the direct limit of the sequence
\begin{equation}\label{direct limit}
G_0 \overset{\varepsilon_0} \twoheadrightarrow G_1 \overset{\varepsilon_1} \twoheadrightarrow G_2
\overset{\varepsilon_2} \twoheadrightarrow \ldots
\end{equation}

For each $n \in \mathbb{N}$, define $S_n$ to be the generating set $\{a_0,t\}$ of $G_n$. In this framework, the key result in the construction from \cite{O-O-S} is the following lemma:

\begin{lem}\label{growth condition}
(\cite[Lem. 3.24]{O-O-S}) The groups $G_n$ are hyperbolic, and, provided the sequence $\textbf{c}$ grows fast enough, the sequence (\ref{direct limit}) (with generating sets $S_n=\{a_0, t\}$ for each $G_n$) satisfies all conditions of Definition \ref{lac hyp defn} so that the direct limit of this sequence, $G(p,\textbf{c})$, is lacunary hyperbolic.
\end{lem}

\begin{rem}\label{injective}
In particular, the proof of this lemma shows that given a finite subset $\mathcal{F}_n$ of $G_n$, any sufficiently large choice of $c_{n+1}$ guarantees that the map $\varepsilon_n : G_n \twoheadrightarrow G_{n+1}$ is injective on $\mathcal{F}_n$.
\end{rem}


\section{Proof of the main result}

Since for any prime number $p$ and any non-decreasing sequence $\textbf{c}$ of natural numbers the group $G(p,\textbf{c})$ is $2$-generated and elementary amenable, the proof of Theorem \ref{main result} is reduced to the following proposition.

\begin{prop}\label{main prop}
Given a prime number $p$, there exists a non-decreasing sequence $\textbf{c}$ of natural numbers $c_1 \leq c_2 \leq \ldots$ such that the group $G(p,\textbf{c})$ defined in Subsection \ref{OOS construction} is MIF.
\end{prop}

To prove this proposition, we will need two lemmas and the following definition.

\begin{defn}
Let $G$ be a non-elementary (i.e. not virtually cyclic) hyperbolic group. Then $G$ contains a unique, maximal finite normal subgroup called the \textit{finite radical} of $G$. (Existence of the finite radical follows from \cite[Prop. 1]{Ol}.)
\end{defn}

\noindent The following lemma is a simplification of \cite[Cor. 1.7]{H-O}.

\begin{lem}\label{AH MIF}
Let $G$ be a non-elementary hyperbolic group with trivial finite radical. Then G is MIF.
\end{lem}

\begin{rem}
It should be noted that trivial finite radical is a necessary condition for a non-elementary hyperbolic group to be MIF. Indeed, for any group $G$ with nontrivial finite normal subgroup $N$ of size $n$ and any $g \in N \backslash \{1\}$, $G$ satisfies the nontrivial mixed identity $[x^{n!}, g]=1$.
\end{rem}

\begin{lem}\label{non-elem}
For each $n \in \mathbb{N}$, the groups $G_n$ defined by (\ref{Gn}) are non-elementary hyperbolic with trivial finite radical.
\end{lem}

\begin{proof}
For each $n \in \mathbb{N}$, the group $G_n$ admits an epimorphism onto the wreath product $\wreath{(\mathbb{Z}/p\mathbb{Z})}{\mathbb{Z}}$ (given by adding the relations $[a_i,a_j]$ for all pairs $i,j \in \mathbb{Z}$) and so $G_n$ is not virtually cyclic. The group $G_n$ is hyperbolic by Lemma \ref{growth condition}.

To see that $G_n$ has no nontrivial finite normal subgroups, first observe that any element $g$ of $G_n$ may be written as $$g=g_1 t^{\alpha_1} g_2 t^{\alpha_2} \ldots g_k t ^{\alpha_k},$$ where $g_1, g_2, \ldots, g_k \in A_n$ and $\alpha_1, \alpha_2 \ldots, \alpha_k \in \mathbb{Z}$. Rewriting the element as $$g=g_1 \left(t^{\beta_1} g_2 t^{-\beta_1}\right)\left((t^{\beta_2} g_3 t^{-(\beta_2)}\right) \ldots \left(t^{\beta_{k-1}} g_k t ^{-\beta_{k-1}}\right) t^{\beta_k},$$ where $\beta_j = \sum\limits_{i=1}^j \alpha_i$, we observe that each of the elements $g_1, \left(t^{\beta_1} g_2 t^{-\beta_1}\right), \ldots, \left(t^{\beta_{k-1}} g_k t ^{-\beta_{k-1}}\right)$ is an element of $A_n$ so that $g=at^{\beta_k}$ for some element $a \in A_n$.

Now observe that if $\beta_k = 0$, then $g \in A_n$, and if $\beta_k \neq 0$, then the image of $g$ in $G_n$ mod the normal closure of $A_n$ is of infinite order, and thus $g$ is of infinite order. Thus all finite-order elements of $G_n$ are contained in $A_n$. So if $H$ is a finite subgroup of $G_n$, then $H \leq \langle a_{-N}, \ldots, a_{N} \rangle$ for some $N \in \mathbb{N}$. But such a subgroup cannot be both normal in $G_n$ and nontrivial, since $H \cap t^{-(2N+1)}Ht^{2N+1} = \{1\}$.
\end{proof}

We can now prove Proposition \ref{main prop}.

\begin{proof}[Proof of Proposition \ref{main prop}]
Set $c_1=1$ and define the set $\mathcal{F}_1=\{1_{G_1}\}$. Fix an enumeration $\{w_i (x)\}_{i \in \mathbb{N}}$ of the elements of $G_0 \ast \langle x \rangle$.

Now given $c_k$, the resulting group $G_k$, and a finite subset $\mathcal{F}_k \subseteq G_k$, choose $c_{k+1}$ and $\mathcal{F}_{k+1}$ as follows. Let $$\pi_k: G_0 \ast \langle x \rangle \rightarrow G_k \ast \langle x \rangle$$ be the extension of the homomorphism $\varepsilon_{k-1} \circ \ldots \circ \varepsilon_0 : G_0 \rightarrow G_k$ given by sending $x \mapsto x$. Let $w_{k,G_k}(x)$ denote the image of $w_k(x)$ under $\pi_k$. If $w_{k,G_k}(x)\neq 1$, first observe that by by Lemma \ref{non-elem}, $G_k$ is non-elementary hyperbolic with trivial finite radical. By Lemma \ref{AH MIF}, $G_k$ is MIF, so there exists $g_k \in G_k$ such that $w_{k,G_k}(g_k) \neq 1$. In this case, add $w_{k,G_k}(g_k)$ to $\mathcal{F}_k$.

Now choose $c_{k+1}$ large enough so that the resulting $\varepsilon_k$ is injective on $\mathcal{F}_k$. (This is possible by Remark \ref{injective}.) Define $$\mathcal{F}_{k+1} = \varepsilon_{k}(\mathcal{F}_k).$$

\noindent After choosing a sequence $\textbf{c}$ in the above manner, It remains to show that the group $G=G(p,\textbf{c})$ is MIF.

For each $k \in \mathbb{N}$ define the homomorphism $$\sigma_k:  G_k \ast \langle x \rangle \rightarrow G \ast \langle x \rangle$$ to be the extension of the natural quotient map $G_k \rightarrow G$ given by sending $x \mapsto x$. To see that $G$ is MIF, first observe that if $w(x) \in (G \ast \langle x \rangle) \backslash \{1\}$, then there exists some $i \in \mathbb{N}$ such that $\sigma_0(w_i(x))=w(x)$, and furthermore, since $w(x)$ is nontrivial, $w_{i,G_i}(x) \neq 1$. By construction, $(\varepsilon_k \circ \ldots \circ \varepsilon_i )(w_{i,G_i}(g_i))$ is in $\mathcal{F}_{k+1} \backslash \{1\}$ for every $k \geq i$, so in particular, $\sigma_i(w_{i,G_i}(g_i)) \neq 1$. Since $\sigma_i(w_{i,G_i}(x))=w(x)$ and $\sigma_i$ is a homomorphism, we have that $\sigma_i(w_{i,G_i}(g_i))=w(\sigma_i(g_i))$. Thus $w(\sigma_i(g_i)) \neq 1$, and so $w(x)=1$ is not a mixed identity on $G$. Since $w(x)$ was arbitrary, this shows that $G$ is MIF.
\end{proof}

\begin{rem}\label{lac hyp rem}
The above construction may be modified so that the resulting group $G(p,\textbf{c})$ is lacunary hyperbolic. To do so, begin by fixing a function $f: \mathbb{N} \rightarrow \mathbb{N}$ so that $n = o(f(n))$. Then, during the step at which $c_{k+1}$ is to be chosen, make the following alteration. Let $\delta_k$ denote the hyperbolicity constant of $G_k$ with respect to the generating set $S_k=\{a_0,t\}$, and let $r(\varepsilon_k)$ denote the injectivity radius of $\varepsilon_k$ with respect to the generating set $S_k$. Now choose $c_{k+1}$ large enough so that in addition to being injective on $\mathcal{F}_k$, the resulting $\varepsilon_k$ satisfies $r(\varepsilon_k) \geq f(\delta_k)$. Define $$\mathcal{F}_{k+1} = \varepsilon_{k}(\mathcal{F}_k)$$ as before, and proceed in the same way. Then, after choosing the sequence $\textbf{c}$, observe that the resulting group $G(p,\textbf{c})$ is lacunary hyperbolic by Lemma \ref{growth condition}.
\end{rem}

To prove Theorem \ref{loc nilp}, we need the following result.

\begin{lem}\label{subnormal}
(\cite[Prop. 5.4(c)]{H-O}) Every nontrivial subnormal subgroup of a MIF group $G$ is also MIF.
\end{lem}

\begin{rem}\label{subnormal rem}
In particular, the proof of \cite[Prop. 5.4(c)]{H-O} notes that if $N$ is a normal subgroup of $G$ satisfying the nontrivial mixed identity $w(x)=1$ for some $w(x) \in N \ast \langle x \rangle$, then for any $n \in N \backslash \{1\}$, $G$ satisfies the nontrivial mixed identity $w([x,n])=1$.
\end{rem}

\begin{proof}[Proof of Theorem \ref{loc nilp}]
Given a prime number $p$, Proposition \ref{main prop} yields a sequence $\textbf{c}$ of natural numbers such that the group $G(p,\textbf{c})$ is MIF. Now observe that the (nontrivial) locally finite $p$-group $A(p,\textbf{c})$ is normal in $G(p,\textbf{c})$, so by Lemma \ref{subnormal}, $A(p,\textbf{c})$ is MIF.
\end{proof}

\section{Notable examples of finitely generated amenable groups which are not MIF}

In this section, we examine two other reasonable candidates for examples of finitely generated MIF amenable groups and explain why they fail to be MIF.

\subsection{The Grigorchuk group}

Given a binary rooted tree $T_2$, we may think of the nodes of $T_2$ as finite binary strings (where the root is represented by the empty string). The \textit{Grigorchuk group} $G$ is defined to be the subgroup of $Aut(T_2)$ generated by elements $a$, $b$, $c$, and $d$ whose actions on binary strings $w \in \{0,1\}^*$ are as follows.
\begin{alignat*}{4}
&a(0w)&&=1w \hspace{1cm}&&a(1w)&&=0w \\
&b(0w)&&=0a(w) &&b(1w)&&=1c(w) \\
&c(0w)&&=0a(w) &&c(1w)&&=1d(w) \\
&d(0w)&&=0w &&d(1w)&&=1b(w)
\end{alignat*}
The Grigorchuk group was initially constructed in \cite{G1} and was shown in \cite{G2} to be the first known example of a finitely generated group with intermediate growth. Notably, it is an example of a group which is amenable but not elementary amenable. (Amenability follows from subexponential growth, while Chou shows in \cite{C} that elementary amenable groups have either polynomial or exponential growth.)

\begin{prop}
The Grigorchuck group satisfies the nontrivial mixed identity $[[[[x,b],d],d],ada]=1$.
\end{prop}

\begin{proof}
To see that the Grigorchuk group satisfies a nontrivial mixed identity, first consider the subgroup $H=\langle b,c,d,aba,aca,ada \rangle$ which is the normal subgroup of index $2$ stabilizing the first level of $T_2$. The subgroup $H$ admits a monomorphism $\varphi: H \hookrightarrow G \times G$ given by sending
\begin{alignat*}{4}
&\varphi(b)&&=(a,c) \hspace{1cm}&& \varphi(aba)&&=(c,a) \\
&\varphi(c)&&=(a,d) \hspace{1cm}&& \varphi(aca)&&=(d,a) \\
&\varphi(d)&&=(1,b) \hspace{1cm}&& \varphi(ada)&&=(b,1) \\
\end{alignat*}
(For references, see \cite[Ch. 8, \#13-14]{dlH}.) Observe that $\varphi$ maps the $H$-conjugates of $ada$ into the first copy of $G$ and the $H$-conjugates of $d$ into the second copy of $G$. Hence, if $K$ is the normal closure in $H$ of the subgroup $\langle d, ada \rangle$, then $\varphi(K)$ is a nontrivial direct product, and since $\varphi$ is injective, $K$ itself decomposes as a nontrivial direct product where the two direct factors are the $H$-conjugates of $d$ and $ada$ respectively. Hence (as in Example \ref{direct product ex}) $K$ satisfies the nontrivial mixed identity $[[x,d],ada]=1$. Since $K$ is subnormal in $G$, we may apply Remark \ref{subnormal rem} twice to obtain that $G$ satisfies the nontrivial mixed identity $[[[[x,b],d],d],ada]=1$.
\end{proof}

\subsection{Akhmedov's construction of amenable groups with infinite girth}

Given a finitely generated group $G$, the \textit{girth} of $G$ is defined to be the infimum of all $n \in \mathbb{N}$ such that for every finite generating set $S$ of $G$, the Cayley graph $\Gamma (G,S)$ contains a cycle of length at most $n$ without self-intersections (see \cite{S}). In \cite{A}, Akhmedov details the construction of a finitely generated amenable group of infinite girth which does not satisfy any nontrivial identity. Such a group is a promising candidate for a finitely generated MIF amenable group not only because it is already identity-free, but also because by \cite[Prop. 5.4(d)]{H-O}, infinite girth is a requisite property for finitely generated MIF groups. However, the group constructed in \cite{A} is a nontrivial (restricted) wreath product, and all such groups satisfy a nontrivial mixed identity. Indeed, for any wreath product (restricted or unrestricted), we have the following.

\begin{prop}
Let $G$ be a wreath product of two nontrivial groups, and let $A \times B$ be any decomposition of the base of the wreath product into a direct product of nontrivial groups $A$ and $B$. Then for any $a \in A \backslash \{1\}$ and any $b \in B \backslash \{1\}$, $G$ satisfies the nontrivial mixed identity $[[[x,a],a],b]=1$.
\end{prop}

\begin{proof}
As in Example \ref{direct product ex}, the base $A \times B$ of the wreath product $G$ satisfies the nontrivial mixed identity $[[x,a],b]=1$. Since the base is normal in $G$, we can apply Remark \ref{subnormal rem} to obtain that $G$ satisfies the nontrivial mixed identity $[[[x,a],a],b]=1$.
\end{proof}

\vspace{1cm}

\noindent \textbf{Bryan Jacobson: } Department of Mathematics, Vanderbilt University, Nashville, TN 37240, U.S.A.\\
E-mail: \emph{bryan.j.jacobson@vanderbilt.edu}


\begin{thebibliography}{99}
 \bibitem{A}
A. Akhmedov, On the girth of finitely generated groups,  \textit{J. Algebra} \textbf{268} (2003), no. 1, 198-208.

 \bibitem{C}
C. Chou, Elementary amenable groups, \textit{Illinois J. Math} \textbf{24} (1980), no. 3, 396-407.

\bibitem{dlH}
P. de la Harpe, Topics in Geometric Group Theory, Univ. Chicago Pr., Chicago, 2000.

 \bibitem{G1}
R.I. Grigorchuk, On Burnside's problem on periodic groups. \textit{Funct. Anal. Appl.} \textbf{14} (1980), no. 1, 53-54 (in Russian).

 \bibitem{G2}
R.I. Grigorchuk, Degrees of growth of finitely generated groups and the theory of invariant means, \textit{Izv.
Akad. Nauk SSSR Ser. Mat.} \textbf{48} (1984), no. 5, 939-985 (in Russian).

 \bibitem{H-O}
M. Hull, D. Osin, Transitivity degrees of countable groups and acylindrical hyperbolicity, \textit{Israel J. Math} \textbf{216} (2016), no. 1, 307-353.

 \bibitem{Ol}
A. Ol'shanskii, On residualing homomorphisms and G-subgroups of hyperbolic groups, \textit{Int. J. Alg. Comp.} \textbf{3} (1993), no. 4, 365-409.

 \bibitem{O-O-S}
A. Ol'shanskii, D. Osin, M. Sapir, Lacunary hyperbolic groups With an appendix by Michael Kapovich and Bruce Kleiner. \textit{Geom. \& Topol.} \textbf{13} (2009), no. 4, 2051-2140.

 \bibitem{S}
S. Schleimer, On the girth of groups. Preprint (2000), available at \url{http://www.warwick.ac.uk/~masgar/math.html}.

\end{thebibliography}
\end{document}